\documentclass[10pt]{amsart}

\usepackage{ mathdots }

\usepackage{psfrag}
\usepackage{amsmath}
\usepackage{amssymb}
\usepackage[english]{babel}
\usepackage{amsthm}
\usepackage{amsfonts}
\usepackage[applemac]{inputenc}
\usepackage{tikz}
\usetikzlibrary{arrows}
\usetikzlibrary{patterns}
\usetikzlibrary{shapes}
\usepackage{mathrsfs}
\usepackage{bbm}

\newtheorem{teo}{Theorem}
\newtheorem{lemma}[teo]{Lemma}

\newtheorem{coro}[teo]{Corollary}
\newtheorem{propo}[teo]{Proposition}

\newcommand{\R}{\mathbb{R}}
\newcommand{\ii}{\mathrm{i}}

\newcommand{\eps}{\epsilon}
\newcommand{\N}{\mathbb{N}}
\newcommand{\C}{\mathbb{C}}
\newcommand{\CP}{\mathbb{C}\textrm{P}}
\newcommand{\RP}{\mathbb{R}\mathrm{P}}

\newcommand{\Z}{\mathbb{Z}}

\newcommand{\Q}{\mathcal{Q}}

\theoremstyle{remark} 
\newtheorem{remark}[]{Remark}

\newtheorem{example}[]{Example}
\newtheorem*{ex}{Example}

\title{Convex pencils of real quadratic forms}

\author{ A. Lerario}
\thanks{SISSA, Trieste}

\begin{document}
\begin{abstract}We study the topology of the set $X$ of the solutions of a system of two quadratic inequalities in the real projective space $\RP^n$ (e.g. $X$ is the intersection of two real quadrics). We give explicit formulae for its Betti numbers and for those of its double cover in the sphere $S^n$; we also give similar formulae for level sets of homogeneous quadratic maps to the plane. We discuss some applications of these results, especially in classical convexity theory. We prove the sharp bound $b(X)\leq 2n$ for the total Betti number of $X$; we show that for odd $n$ this bound is attained only by a singular $X$. In the nondegenerate case we also prove the bound on each specific Betti number $b_k(X)\leq 2(k+2).$
\end{abstract}
\maketitle

\section{Introduction}
In this paper we study the topology of a set $X$ defined in the projective space $\RP^n$ by a system of two quadratic inequalities (equalities are permitted as well). Such a set is described by a quadratic map $q:\R^{n+1}\to \R^2$, i.e. a map whose components are real homogeneous quadratic forms $q_0, q_1,$ and a convex polyhedral cone $K\subset \R^2$. Given these two data we may define $X$ as:
$$X=\{[x]\in \RP^n\,|\, q(x)\in K\}$$
More generally one can consider the set of the solutions of a system of $k$ quadratic inequalities, e.g. the intersection of $k$ quadrics in $\RP^n.$ The idea for studying  such an object is to exploit a kind of duality between the inequalities defining $X$ and the variables appearing in them. We explain this idea now. Let
	$$q:V\to W$$ be a quadratic map between two vector spaces and for every covector $\eta$ in $W^{*}$ consider the composition $\eta q$: it is a quadratic form over $V$. As we let $\eta$ vary we can reconstruct the map $q$ itself from the various $\eta q.$ As long as we are concerned with the topology of $X$, we can replace this complicated object with a simpler one: for each $\eta$ in $W^{*}$ we consider only the positive inertia index $\ii^{+}(\eta q)$, i.e. the maximal dimension of a subspace of $V$ on which $\eta q$ is positive definite. Thus, instead of dealing with a map with values in the space of quadratic forms, we have a function to the natural numbers. We let $K^{\circ}$ be the polar cone of $K$ and $S^{k-1}$ be the unit sphere in $W$ (with respect to any scalar product). We define the following sets:	$$\Omega=K^{\circ}\cap S^{k-1}\quad \textrm{and}\quad \Omega^{j}=\{\omega \in \Omega \, |\, \ii^{+}(\omega q)\geq j\}, \quad j\in \N$$ 
The spirit of the mentioned duality is in this procedure of replacing the original framework with the above filtration $\Omega^{n+1}\subseteq \cdots \subseteq \Omega^{0}$. This idea was first introduced by Agrachev in the paper \cite{Agrachev2} to study the topology of the double cover of the complement of $X$ under some regularity assumption; the extensive study of the same idea was the subject of \cite{AgLe}, where the topology of $X$ itself was studied, removing also the nondegeneracy assumption. The technique of \cite{Agrachev2} and \cite{AgLe} is that of spectral sequences, a subject that despite its importance in topology is still considered to be for specialists. The aim of this paper is thus, at first, to give an account of the previous theory avoiding spectral sequences at all. This is possible in the case of \emph{two} quadratic inequalities and the reason fort that is essentially the following observation: in the general case the above filtration is a ``continuous'' object, in the sense that the sets of points where the function $\ii^{+}$ changes its value is an algebraic subset $Z$ of $S^{k-1};$ in the case of a quadratic map to the \emph{plane}, $Z$ reduces to a finite number of points on $S^{1}$ and our object becomes something ``discrete''.\\
Moving to topology, the crucial point is that there is a relation between the geometry of the previous filtration and that of $X$ itself. To give an example, in the case $X$ is the intersection of quadrics in $\RP^{n},$ the following formula relates the Euler characteristic of it with that of the previously defined sets (here $K$ is the zero cone and the components of $q$ are given by the quadratic forms defining $X$):
 \begin{equation}\label{eq:euler}(-1)^{n}\chi(X)=\chi(S^{n})+\sum_{j=0}^n(-1)^{j+1}\chi(\Omega^{j+1}).\end{equation}
 In the case $X$ is the intersection of \emph{two} quadrics the topology of the sets $\Omega^{j},\, j\in \N,$ being semialgebraic subsets of $S^{1}$, is very easy to compute (such sets are finite union of points and arcs).
 \begin{ex}[The bouquet of three circles]Consider the map $q:\R^{4}\to \R^{2}$ defined by 
 	$$x\mapsto (x_{1}^{2}+2x_{0}x_{2}-x_{3}^{2}, x_{1}x_{2})$$
	and the zero cone in $\R^2$.
The subset $X$ of $\RP^3$ defined by $\{q=0\}$ consists of two projective line and an ellipse meeting at one point; this set
is homeomorphic to a bouquet of three circles. Associating to a
quadratic form a symmetric matrix by means of a scalar product, the
family $\eta q$ for $\eta \in \Omega=S^1$ is represented by the
matrix:
$$\eta S=\left(\begin{array}{cccc}
0&0&\eta_0&0\\
0&\eta_0&\eta_1&0\\
\eta_0&\eta_1&0&0\\
0&0&0&-\eta_{0}
\end{array}\right)\quad \eta=(\eta_0,\eta_1)\in S^1$$
The determinant of this matrix vanishes only at the points $\omega=(0,1)$
and $-\omega=(0,-1);$ outside of these points the index function must
be locally constant. Then it is easy to verify that $\ii^+$ equals $2$
everywhere except at this two points, where it equals $1$:
       $$\Omega^1=S^1,\quad \Omega^2=S^1\backslash\{\omega,-\omega\},\quad
\Omega^3=\Omega^4=\emptyset.$$
Applying formula (\ref{eq:euler}) to this example gives:
$$-\chi(X)=\chi(S^{3})+\chi(S^{1})-\chi(S^{1}\backslash\{\omega, -\omega\})=2.$$
\end{ex}

A formula similar to (\ref{eq:euler}) holds in the case of the intersection $Y$ of two quadrics on the sphere (since $Y$ double covers $X$ we simply have to multiply the right hand side by two). In this case we can be even more precise; if we set
	$$Y=q^{-1}(K)\cap S^{n}$$
the following holds for the reduced Betti numbers\footnote{from now on we work with $\Z_{2}$ coefficients.} of $Y$:
\begin{equation}\label{eq:spherical}\tilde b_{k}(Y)=b_{0}(\Omega^{n-k},\Omega^{n-k+1})+b_{1}(\Omega^{n-k-1},\Omega^{n-k}),\quad k<n-2\end{equation}
This formula was essentially proved by Agrachev in \cite{Agrachev2} under some smoothness hypothesis; we will give a proof of it with no nondegeracy assumption in the Appendix. It will be the only part where spectral sequences will be used in this paper; we have added this proof both for the sake of completeness and to give the reader a soft flavor of the spectral sequences arguments. To stress the difference with the general case, we must say the analogous formulae for the Betti numbers of intersection of more than two quadrics only give bounds for them.\\
As pointed out to the author by the referee, similar results were obtained in the nondegenerate case in \cite{Lo}; there it is given even a classification up to diffeomorphisms of the set of smooth spherical intersection of two quadrics.\\
The formula to compute the Betti numbers of $X$ is slightly different from (\ref{eq:spherical}) and the reason for that lies in the structure of the homology of $\RP^{n}$ which is richer than the one of the sphere. Before going into it we will give the formula for the Betti numbers of $\RP^{n}\backslash X$:
\begin{equation}\label{eq:complproj}b_{k}(\RP^{n}\backslash X)=b_{0}(\Omega^{k+1})+b_{1}(\Omega^{k}),\quad k\in \N\end{equation}
With (\ref{eq:complproj})  at our disposal, we can already give some applications. The second goal of this paper is to show how classical theorems in convexity theory can be easily derived from this setting. For example the quadratic convexity theorem (the image of the sphere $S^{n}$ under a homogeneous quadratic map $q:\R^{n+1}\to \R^{2}$ with $n\geq 2$ is a convex subset of the plane) is obtained using (\ref{eq:complproj}). Notice that  this result is certainly false in the case the target space is three dimensional, as the map $x\mapsto(x_{0}x_{1},x_{0}x_{2},x_{1}x_{2})$ shows.\\
Moving back to the Betti numbers of $X$, in order to get them from (\ref{eq:complproj}) we need to compute the rank of the homomorphism induced on the homology by the inclusion $i:X\hookrightarrow \RP^{n}$ (this computation was not necessary for the sphere because of Alexander duality). If we let $\mu$ be the maximum of $\ii^{+}$ on $\Omega$ we have:
	\begin{equation}\label{eq:rank}\textrm{rk}(i_{*})_{k}=b_{0}(C\Omega,\Omega^{k+1}), \quad k\neq n- \mu\end{equation}
where $C\Omega=(\Omega \times [0,1])/(\Omega \times \{0\})$ is the topological cone of $\Omega$ (see \cite{Hatcher}). The critical case $k=n-\mu$ is more subtle and we need extra information. For this purpose we introduce the bundle $L_{\mu}\to \Omega^{\mu}$ whose fiber over the point $\eta\in \Omega^{\mu}$ equals $\textrm{span}\{x\in \R^{n+1} \, |\, \exists \lambda >0 \,\textrm{ s.t.}\, (\eta Q)x=\lambda x\}$ and whose vector bundle structure is given by its inclusion in $\Omega^{\mu}\times \R^{n+1}.$ The extra information we need is the first Stiefel-Whitney class of $L_{\mu}:$
	$$w_{1}(L_{\mu})\in H^{1}(\Omega^{\mu}).$$ 
Once we have this data we can compute also the rank of $(i_{*})_{n-\mu}$:
	\begin{equation}\label{eq:sw}\textrm{rk}(i_{*})_{n-\mu}=1 \quad \textrm{iff}\quad w_{1}(L_{\mu})=0.\end{equation}
Consider now the table $E=(e^{j,j})_{i,j\in \Z}$ whose nonzero part is the following:
$$
E'=\begin{array}{|c|c|c|}
\hline
1&0&0\\
\hline
\vdots&\vdots&\vdots\\
\hline
1&0&0\\

\hline
c&0&0\\
\hline
0& b_{0}(\Omega^{\mu})-1&d\\
\hline
\vdots&\vdots&\vdots\\

\hline
0&b_{0}(\Omega^{1})-1&b_{1}(\Omega^{1})\\
\hline
\end{array}
$$
where  $c=e^{0,\mu}$ and we have $(c,d)=(1,b_{1}(\Omega^{\mu}))$ if $w_{1,\mu}=0$ and $(c,d)=(0,0)$ otherwise. In term of the previous table it is easy to write the formula for the Betti numbers of $X$: if $\mu=n+1$ then $X$ is empty; in the contrary case for every $k\in \Z$ we have:
\begin{equation}\label{eq:proj}b_{k}(X)=e^{0,n-k}+e^{1,n-k-1}+e^{2,n-k-2}.\end{equation}
Notice that according to what we already stated $\textrm{rk}(i_{*})_{k}=e^{0,n-k}$. The previous table is essentially the way to plug in Alexander-Pontryiagin duality in formula (\ref{eq:proj}).
\begin{ex}[The bouquet of three circles; continuation]
In this case the previous table is the following:
$$E=\begin{array}{|c|c|c}
1&0&0\\
1&0&0\\
0&1&0\\
0&0&1\\
\hline
\end{array}$$
This is because $\mu=2$ and $\Omega^{2}$ is not the whole $S^{1}$ (thus $w_{1}(L_{2})=0$); in particular we have:
	$$b_{0}(X)=e^{0,3}+e^{1,2}+e^{2,1}=1, \quad b_{1}(X)=e^{0,2}+e^{1,1}+e^{2,0}=3.$$
Notice also that $\textrm{rk}(i_{*})_{1}=e^{0,2}=1$, in accordance with the fact that $X$ contains a projective line.
\end{ex}
For the reader familiar with spectral sequences we can say that the previous table gives the ranks of the second term of a spectral sequence converging to the homology of $X$; the class $w_{1}(L_{\mu})$ gives the second differential and for the case of two quadrics that's enough (the spectral sequence degenerates at the third step); otherwise higher differentials have to be calculated (the reader is referred to \cite{AgLe} for a detailed discussion on the general case); this evidence again confirms the difference with the case of more than two quadrics.\\Thus in a certain sense we are stating our results using the language of spectral sequences, without actually using them.\\
Incidentally we notice that it is possible to ``send the number of variables to infinity'' and our procedure is stable with respect to this limit: formulae similar to the previous ones hold for the set of the solutions of a system of two quadratic inequalities on the infinite dimensional sphere; the interested reader is referred to \cite{Le}.\\
The third, and final goal of this paper is to use the previous theory for studying the homological complexity of $X$. Consider the problem of bounding the sum of the Betti numbers of $X$, such number is denoted by $b(X)$ and is called the \emph{homological complexity}. The classical Oleinik-Petrovskii-Thom-Milnor-Smith bound gives for the intersection of $k$ quadrics in $\RP^n$ the estimate $b(X)\leq O(k)^n$. Another manifestation of the mentioned duality between the number of quadratic equations and the number of variables is a classical result due to Barvinok \cite{Barvinok}, which states that the two numbers $k$ and $n$ can be interchanged in the previous bound, giving $b(X)\leq n^{O(k)}$ (the implied constant is at least two). In the case $X$ is the intersection of two quadrics in $\RP^n$ we will see that indeed we have the sharp estimate:
\begin{equation}\label{eqbound}b(X)\leq 2n\end{equation}
It is interesting at this point, as suggested by the referee, to decide whether the maximal complexity is attained by a smooth $X$ or not. It turns out that for even $n$ this is the case, while for odd $n$ it is not. 
\begin{ex}Consider the algebraic set $X$ in $\RP^3$ defined by: $$X=\{x_0^2-x_2^2=x_1^2-x_3^2=0\}$$ Then $X$ equals the union of the four projective lines $L_1=\{x_0=x_2, x_1=-x_3\}, L_2=\{x_0=-x_2, x_1=x_3\}, L_3=\{x_0=-x_2, x_1=-x_3\}$ and $L_4=\{x_0=x_2, x_1=x_3\}$; an easy computations of the combinatorics of these lines shows that $X$ is homotopically equivalent to the wedge of $5$ circles and its total Betti number is $6$. Similarly the algebraic set $X_\C$ defined by the same equations as $X$ in $\CP^3$ is homotopically equivalent to the wedge of $5$ spheres and $b(X_\C)=6.$\\On the other hand by the adjunction formula the complete intersection of two quadrics in $\CP^3$ has genus one and total Betti number $4$; thus by Smith's inequality\footnote{This inequality is the statement that the homological complexity of a real algebraic variety $X_\R$ is always less or equal to the one of its complex part $X_\C$} the  for every \emph{smooth} intersection $X'$ of two quadrics in $\RP^3$ we have $b(X')\leq 4.$ \end{ex}
We will provide in the paper a general family of examples, smooth for even $n$ and singular for odd ones, attaining the bound (\ref{eqbound}). We notice that this family of examples, together with Smith's inequality, proves that for the intersection of two complex quadrics in an odd dimensional complex projective space the maximal complexity is attained by a singular intersection.
In the case of a system of two quadratic \emph{inequalities} a similar bound can be produced.\\
We conclude the paper with an estimate on each single Betti number; for this estimate though we require $X$ to be a \emph{nonsingular} intersection of two quadrics:
\begin{equation}\label{eqboundsingle}b_k(X)\leq 2(k+2)\end{equation}
What is interesting about the bound (\ref{eqboundsingle}) is that for each specific Betti number it does not depend on $n$.\\
The paper is organized as follows. In Section 2 we introduce the necessary notation. In Section 3 we make a general construction to study one dimensional families of quadratic forms. Section 4 is devoted to prove formulae  (\ref{eq:complproj}) and (\ref{eq:proj}); the statement (\ref{eq:euler}) about the Euler characteristic directly follows from (\ref{eq:complproj}). In Section 5 we discuss the proof of some classical results using the introduced technique. Section 6 deals with the homological complexity and contains the proofs of (\ref{eqbound}) and (\ref{eqboundsingle}). In the Appendix we give a proof of formula (\ref{eq:spherical}) as well as a similar one for the level sets of a homogeneous quadratic map.\section*{acknowledgments}
The author is very grateful to his teacher A. A. Agrachev for the remarkable contribution to this work; it can be considered as the natural development of many of his ideas. The author wishes to thanks also the anonymous referee, for her/his patience and for the many helpful suggestions which improved both the presentation and the results.

\section{Notation} For the rest of the paper we will assume that the following two data will be specified:
(i) a closed convex cone $K$ in $\R^2$; (ii) a pair $q_0,q_1$ of real homogeneous quadratic form in $n+1$ variables. The set we will be interested in is:
$$X=\{[x]\in \RP^n\,|\, (q_0(x), q_1(x))\in K\}.$$
In the particular case $K=\{0\}$ the set $X$ is the intersection of the two quadrics $\{q_0=0\}$ and $\{q_1=0\}$; in the general case $X$ is the set of the solutions of a system of quadratic inequalities (the inequalities being given by the presentation of the polyhedral cone $K$).\\
If $V$ is a finite dimensional real vector space, the symbol $\Q(V)$ will denote the set of real quadratic forms over $V;$ in the case $V=\R^k$ we will simply write $\Q(k)$.\\
For every $p$ in $\Q(V)$ the positve inertia index $\ii^+(q)$ is the maximal dimension of a subspace $W\subset V$ such that $p|_W>0.$ Once a scalar product has been fixed the following gives a rule to identify each $q$ in $\Q(k)$ with a $Q$ in $\textrm{Sym}_k(\R)$, the set of real symmetric matrices of dimension $k$:
$$q(x)=\langle x, Qx\rangle \quad \forall x\in \R^k.$$
Under this (linear) identification the inertia index of $q$ equals the number of positive eigenvalues of $Q.$\\
In the case $q=(q_0,q_1)$ is a pair of quadratic forms in $\Q(V)$ and $\omega=(\omega_0,\omega_1)$ is a point in $(\R^2)^*$, we will write $\omega q$ for the quadratic form $\omega_0 q_0 + \omega_1 q_1;$ similarly if $Q=(Q_0, Q_1)$ is a pair of symmetric matrices in $\textrm{Sym}_k(\R)$ we will write $\omega Q$ for the symmetric matrix $\omega_0 Q_0+\omega_1 Q_1.$ This notation comes from the fact that we may interpret $q=(q_0, q_1)$ as a quadratic \emph{map} from $V$ to $\R^2$ and $\omega q$ is simply the composition of this map with the covector $\omega.$ The map $q$ itself defines also a map $q^*:(\R^2)^*\to \Q(V)$ given by $\omega \mapsto \omega q.$ These are all ways of reformulating the same fact.\\
Given $q=(q_0,q_1)$ and $K$ we will need to define the sets:
$$\Omega=K^\circ \cap S^1\quad \textrm{and}\quad\Omega^j=\{\omega \in \Omega\,|\, \ii^+(\omega q)\geq j\},\quad j\in \N$$
where $K^\circ=\{\eta \in (\R^2)^*\, | \eta(y)\leq 0 \textrm{ for all $y$ in $K$}\}$\\
If $A$ is a topological space we will denote by $H_*(A)$ its homology with $\Z_2$ coefficients and by $b_*(A)$ its Betti numbers with $\Z_2$ coefficients. In our case will aways be defined the number $b(A)$: it is the sum of the Betti numbers of $A$ with $\Z_2$ coefficients (this quantity is usually referred as the \emph{homological complexity} of $A$). Similar definitions apply for a pair of spaces $(A, B).$
\section{A preliminary construction}

Let now $\Omega$ be an arbitrary closed semialgebraic subset of the unit circle $S^{1}$ and $f:\Omega \to \Q(n+1)$ be a semialgebraic map. The map $f$ describes a family of quadratic forms varying semialgebraically w.r.t. $\omega\in \Omega.$ We define the semialgebraic function
\begin{equation}\label{defF}F:\Omega\times \RP^n\to \R\quad \quad F(\omega,[x])=f(\omega)(x).\end{equation}The fact that this function is well defined follows from the fact that $f(\omega)$ is a quadratic form, homogeneous of even degree, and thus  $f(\omega)(x)=f(\omega)(-x);$ the semialgbebraicity of $F$ is obvious. Consider the semialgebraic set:
	\begin{equation}\label{defC}C=\{F\geq 0\}.\end{equation}
Since the projection $p_{1}:C\to \Omega$ is a semialgebraic map, then by Hardt's triviality theorem (see \cite{BCR}) there exixts a finite semialgebraic partition $\Omega=\coprod S_{l}$ such that  $p_{1}$ is trivial over each $S_{l}$. The semialgebraic subsets of $\Omega$ are union of points and intervals (arcs); thus there exist a finite number of points $\{\omega_{\alpha}\}_{\alpha\in A}$ and a finite number of open arcs $\{I_{\alpha\beta}\}_{\alpha, \beta \in A}$ such that $C$ is the disjoint union of the inverse image under $p_{1}$ of them; moreover $p_{1}$ is trivial over each of these subsets of $\Omega.$ For each $\omega \in \Omega$ and $k\in \N$ we define:
	\begin{equation}\label{genomega}a(\omega)= n-\ii^{-}(f(\omega))\quad \textrm{and}\quad\Omega_{n-k}=\{\omega \in \Omega \, |\, a(\omega)\geq k\}\end{equation}
	In other words the set $\Omega_{n-k}$ equals the closed semialgebraic subset of $\Omega$ consisting of those $\omega$ for which $\ii^-(f(\omega))\leq n-k.$\\
	 Given a quadratic form $q$ we define $P^+(q)$ to be the positive eigenspace of the corresponding symmetric matrix $Q.$ Using this notation we clearly have that $p_{1}^{-1}(\eta)$ deformation retracts to $P^{+}(f(\eta))\simeq \RP^{a(\eta)}.$\\
	 Consider now the topological space
	$$S=\{(\omega, [x])\in \Omega \times \RP^{n}\, |\, [x]\in P^{+}(\omega)\}.$$
	\begin{lemma} The inclusion $S\hookrightarrow C$ is a homotopy equivalence; indeed $C$ deformation retracts to $S$.
\begin{proof} For every $\alpha\in A$ let $U_{\alpha}$ be a closed neighborhood of $\omega_{\alpha}$ such that the inclusion $P^{+}(\omega_{\alpha})\hookrightarrow C|_{U_{\alpha}}$ is a homotopy equivalence (such a neighborhood exists by triangulability of the semialgebraic function $f:(\omega, [x])\mapsto \textrm{dist}(\omega, \omega_{\alpha})$ and noticing that the inclusion $P^{+}(\omega_{\alpha})\hookrightarrow \{f=0\}$ is a homotopy equivalence). If $U_{\alpha}$ is sufficiently small, then $S|_{U_{\alpha}}$ deformation retracts to $P^{+}(\omega_{\alpha}):$ since the eigenvalues of $f(\omega)$ depend continuously on $\omega$ and $\textrm{dim}(P^+(\omega_{\alpha}))\geq \textrm{dim}(P^+(\omega))$ for $\omega$ sufficiently close to $\omega_{\alpha}$, the deformation retraction is performed simply by sending   each $P^{+}(\omega)$ to $\lim_{\omega \to \omega_{\alpha}}P^{+}(\omega)\subseteq P^{+}(\omega_{\alpha}).$ Now we have that $P^{+}(\omega_{\alpha})\hookrightarrow S|_{U_{\alpha}}$ and $P^{+}(\omega_{\alpha})\hookrightarrow C|_{U_{\alpha}}$ are both homotopy equivalences; since the second one is the composition $P^{+}(\omega_{\alpha})\hookrightarrow S|_{U_{\alpha}}\hookrightarrow C|_{U_{\alpha}}$ then $S|_{U_{\alpha}}\hookrightarrow C|_{U_{\alpha}}$ also is a homotopy equivalence. Since $(C|_{U_{\alpha}},S|_{U_{\alpha}})$ is a CW-pair, then the previous homotopy equivalence implies $C|_{U_{\alpha}}$ deformation retracts to $S|_{U_{\alpha}}$ (see \cite{Hatcher}).\\
Let now $W=(\cup_{\alpha}V_{\alpha});$ since $C|_{W^{c}}$ is a locally trivial fibration, then clearly it deformation retracts to $S|_{W^{c}}$; since each $V_{\alpha}$ is closed, then $C$ deformation retracts to $C|_{W}\cup S|_{W^{c}}$. Since the deformation retraction of each $C|_{U_{\alpha}}$ fixes $S|_{U_{\alpha}}$ and $\textrm{Cl}(S|_{W^{c}})\cap C|_{W}\subseteq S $ then all these deformation retractions match together to give de desired one of $C$ to $S.$\end{proof}
	\end{lemma}
The following lemma describes the cohomology of $C.$
	\begin{lemma}\label{one}$H^{k}(C)\simeq H_{0}(\Omega_{n-k})\oplus H_{1}(\Omega_{n-k+1}).$
\begin{proof}We only give a sketch; the rigorous details are left to the reader. A spectral sequence argument, using the Leray sheaf, will be provided in the Appendix 1.\\
We can give a cellular structure to $S$ in the following way: for every $\omega_{\alpha}$ such that $a(\omega_{\alpha})\geq k$ we place a $k$-dimensional cell $e_{\alpha}^{k}$ representing a $k$-dimensional cell of $P^{+}(\omega_{\alpha});$ for every arc $I_{\alpha\beta}$ such that $a(\omega)\geq k-1$ for every $\omega \in I_{\alpha \beta}$ we place another $k$-dimensional cell $e_{\alpha\beta}^{k}$ representing a $k$ dimensional cell of $S|_{I_{\alpha\beta}}.$ In this way, working with $\Z_{2}$ coefficients we have:
	$$\partial e_{\alpha}^{k}=0\quad \textrm{and}\quad \partial e_{\alpha\beta}^{k}=e_{\alpha}^{k-1}+e_{\beta}^{k-1}$$
and the statement follows now from cellular homology  and Leray-Hirsch theorem (see \cite{Hatcher}).
\end{proof}
	\end{lemma}
	
Together with the closed set $C$ we will be interested in the set $B$ defined by taking the strict inequality in (\ref{defC}):
\begin{equation}\label{defB}B=\{F>0\}\end{equation}
Similarly to (\ref{genomega}) we define the open semialgebraic sets
$$\Omega^k=\{\omega\in \Omega\,|\, \ii^+(f(\omega))\geq k\},\quad k\in \N$$
The cohomology of $B$ is given by the following lemma.
\begin{lemma}\label{cooB}$H^k(B)\simeq H_0(\Omega^{k+1})\oplus H_1(\Omega^{k}).$
\begin{proof}
Since $F$ is a proper semialgebraic function, then for $\epsilon>0$ small enough the inclusion
$$B\hookrightarrow\{F\geq \eps\}$$
is a homotopy equivalence. On the other hand the set $\{F\geq \eps\}$ is homeomorphic to the set $C$ defined as before but for the map $f_\eps=f-\eps p$ for a positive definite form $p,$ i.e. setting $F_\eps(\omega,[x])=F(\omega, [x])-\eps p(x)$ we have 
$$\{F\geq \eps\}\simeq \{F_\eps\geq 0\}$$ Thus by lemma \ref{one} the cohomology of $B$ is isomorphic to
$$H^k(B)\simeq H_0(\Omega_{n-k}(\eps))\oplus H_1(\Omega_{n-k+1}(\eps))$$
where 
$$\Omega_{n-k}(\eps)=\{\omega\in \Omega\,|\, \ii^+(\eps p-f(\omega))\leq n-k\}$$
and the lemma follows from the following proposition.
\end{proof}
\end{lemma}

\begin{propo}\label{union}For every positive definite form $p\in \Q(n+1)$  and for every $\eps>0$ sufficiently small
$$ H_*(\Omega_{n-k}(\eps))\simeq H_*(\Omega^{k+1})$$
where $\Omega_{n-k}(\eps)=\{\omega \in \Omega\,|\, \ii^-(f(\omega)-\eps p)\leq n-k\}.$
\begin{proof}
Let us first prove that $\Omega^{j+1}=\bigcup_{\eps>0}\Omega_{n-j}(\eps).$\\
Let $\omega \in \bigcup_{\eps>0}\Omega_{n-j}(\eps);$ then there exists $\overline{\eps}$ such that $\omega\in \Omega_{n-j}(\eps)$ for every $\eps<\overline{\eps}.$ Since for $\eps$ small enough $$\ii^{-}(f(\omega)-\eps p)=\ii^{-}(f(\omega))+\dim (\ker (f(\omega))$$ then it follows that $$\ii^{+}(f(\omega))=n+1-\ii^{-}(f(\omega))-\dim (\ker f(\omega))\geq j+1.$$ Viceversa if $\omega \in \Omega^{j+1}$ the previous inequality proves $\omega\in \Omega_{n-j}(\eps)$ for $\eps$ small enough, i.e. $\omega \in \bigcup_{\eps>0}\Omega_{n-j}(\eps).$\\
Notice now that if $\omega\in \Omega_{n-j}(\eps)$ then, eventually choosing a smaller $\eps$, we may assume $\eps$ properly separates the spectrum of $\omega$ and thus, by continuity of the map $f$, there exists $U$ open neighborhood of $\omega$ such that $\eps$ properly separates also the spectrum\footnote{The reader is advised to see \cite{Kato} for a detailed discussion of the regularity of the eigenvalues of a family of symmetric matrices} of $f(\eta)$ for every $\eta \in U$. Hence every $\eta \in U$ also belongs to $\Omega_{n-j}(\eps)$. From this consideration it easily follows that each compact set in $\Omega^{j+1}$ is contained in some $\Omega_{n-j}(\eps)$ and thus $$\varinjlim_{\eps}\{H_{*}(\Omega_{n-j}(\eps))\}=H_{*}(\Omega^{j+1}).$$ It remains to prove that the topology of $\Omega_{n-j}(\eps)$ is definitely stable in $\eps$ going to zero. Consider the semialgebraic compact set $S_{n-j}=\{(\omega, \eps)\in S^{k}\times [0, \infty)\, |\, \ii^{-}(f(\omega)-\eps p)\leq n-j\}$. By Hardt's triviality theorem we have that the projection $(\omega, \eps)\mapsto \omega$ is a locally trivial fibration over $(0,\eps)$ for $\eps$ small enough; from this the conclusion follows.
\end{proof}
\end{propo}

\section{Formulae for the Betti numbers}
We start by proving the following.
\begin{teo}\label{projective}
$b_{k}(\RP^{n}\backslash X)=b_{0}(\Omega^{k+1})+b_{1}(\Omega^{k})$ for every $k\in\N.$
\begin{proof} 
Consider the map $f:\Omega\to \Q(n+1)$ defined by $\omega\mapsto \omega q$ and the map $F:\Omega \times \RP^n \to \R$ defined as in (\ref{defF}). The projection $p_2:\Omega\times \RP^n\to \RP^n$ on the second factor restricts to a homotopy equivalence $$p_2|_{B}:B\simeq \RP^n\backslash X$$
(the image of $p_2|_{B}$ is $\RP^n\backslash X$ because $K^{\circ \circ}=K$ and it is a homotopy equivalence because the fibers are contractible). The result follows from lemma \ref{cooB}.
\end{proof}
\end{teo}

To compute the homology of $X$ we need to know the map induced by the inclusion $c:\RP^{n}\backslash X\to\RP^{n}$ on the cohomology.

\begin{propo}
Set $\mu=\max_{\omega \in\Omega}\ii^{+}(\omega).$ Then for $k\leq \mu-1$
$$H^{k}(\RP^{n})\stackrel{c^{*}}{\rightarrow}H^{k}(\RP^{n}\backslash X)$$
is injective and for $k\geq \mu+1$ is zero.
\end{propo}
Notice that the case $k=\mu$ is excluded from this statement: it deserves a special treatment.
\begin{proof}
Consider the commutative diagram of maps
$$\begin{tikzpicture}[xscale=2.5, yscale=2]

    \node (A0_0) at (0, 0) {$\RP^{n}\backslash X$};
    \node (A1_0) at (1, 0) {$ \RP^{n}$};
    \node (A0_1) at (0, 1) {$B$};
    \node (A1_1) at (1, 1) {$\Omega \times \RP^{n}$};
    \path (A0_0) edge [->] node [auto] {$c$} (A1_0);
    \path (A0_1) edge [->] node [auto,swap ] {$p_{2|_{B}}$} (A0_0);
    \path (A0_1) edge [->] node [auto] {$i$} (A1_1);
    \path (A1_1) edge [->] node [auto] {$p_{2}$} (A1_0);
      \end{tikzpicture}
$$
Since $p_{2}|_{B}$ is a homotopy equivalence, then $c^{*}=i^{*}\circ p_{2}^{*}.$ If $k\leq \mu-1,$ then $\Omega^{k+1}\neq \emptyset;$ thus let $\eta \in\Omega^{k+1}.$ Then $p_{1}^{-1}(\eta)\cap B$ deformation retracts to $\{\eta\}\times P^{d_{\eta}},$ where $P^{d_{\eta}}$ is a projective space of dimension $d_{\eta}=\ii^{+}(\eta)-1\geq k;$ in particular the inclusion $P^{d_{\eta}}\stackrel{i_{\eta}}{\rightarrow}\RP^{n}$ induces an isomorphism on the $k$-th cohomology group. The following factorization of $i_{\eta}^{*}$ concludes the proof of the first part (all the maps are the natural ones): 
$$\begin{tikzpicture}[xscale=2.5, yscale=2]

    \node (A0_0) at (0, 0) {$H^{k}(\Omega \times \RP^{n})$};
    \node (A1_0) at (1, 0) {$H^{k}(\overline{B})$};
    \node (A0_1) at (0, 1) {$H^{k}(\RP^{n})$};
    \node (A1_1) at (1, 1) {$H^{k}(P^{d_{\eta}})$};
    \path (A0_0) edge [->] node [auto] {$$} (A1_0);
    \path (A0_1) edge [->] node [auto,swap ] {$$} (A0_0);
    \path (A0_1) edge [->] node [auto] {$i_{\eta}^{*}$} (A1_1);
    \path (A1_0) edge [->] node [auto] {$$} (A1_1);
      \end{tikzpicture}
$$
For the second statement simply observe that for $k\geq \mu+1$ we have $\Omega^{k}=\emptyset$ and thus 
$$H^{k}(\RP^{n}\backslash X)\simeq H^{0}(\Omega^{k+1})\oplus H^{1}(\Omega^{k})=0.$$
\end{proof}
It remains to study $H^{\mu}(\RP^{n}\backslash X)\to H^{\mu}(\RP^{n}).$ For this purpose we introduce the bundle $L_{\mu}\to \Omega^{\mu}$ whose fiber at the point $\eta\in \Omega^{\mu}$ equals $\textrm{span}\{x\in \R^{n+1} \, |\, \exists \lambda >0 \,\textrm{ s.t.}\, (\eta Q)x=\lambda x\}$ and whose vector bundle structure is given by its inclusion in $\Omega^{\mu}\times \R^{n+1}.$ We let $w_{1,\mu}\in H^{1}(\Omega^{\mu})$ be the first Stiefel-Whitney class of $L_{\mu}.$ We have the following result.

\begin{propo} $\textrm{rk}(c^{*})_{\mu}=0\quad \iff \quad w_{1,\mu}=0.$
\end{propo}
\begin{proof}
In the case $\Omega^{\mu}\neq S^{1},$ then clearly $w_{1,\mu}$ is zero and also $\textrm{rk}(c^{*})_{\mu}$ is zero since $H^{\mu}(\RP^{n}\backslash X)=0.$ If $\Omega^{\mu}=S^{1},$ then $\ii^{+}$ is constant and we consider the projectivization $P(L_{\mu})$ of the bundle $L_{\mu}.$ In this case it is easily seen that the inclusion
$$P(L_{\mu})\stackrel{\lambda}{\hookrightarrow}B$$
is a homotopy equivalence and, since $\textrm{rk}(c^{*})=\textrm{rk}(i^{*}\circ p_{2}^{*})$ we have $\textrm{rk}(c^{*})=\textrm{rk}(\lambda^{*}\circ i^{*}\circ p_{2}^{*}).$ Let us call $l$ the map $p_{2}\circ i \circ \lambda;$ then $l:P(L_{\mu})\to \RP^{n}$ is a map which is linear on the fibres and if $y\in H^{1}(\RP^{n})$ is the generator, we have by Leray-Hirsch
$$H^{*}(P(L_{\mu}))\simeq H^{*}(S^{1})\otimes\{1,l^{*}y,\ldots, l^{*}y^{\mu-1}\}.$$
By the Whitney formula we get
$$l^{*}y^{\mu}=w_{1,\mu}(l^{*}y)^{\mu-1}$$
which proves $(c^{*})_{\mu}$ is zero iff $w_{1,\mu}=0.$
\end{proof}

Collecting together Theorem \ref{projective} and the previous two propositions allows us to split the long exact sequence of the pair $(\RP^{n},\RP^{n}\backslash X)$ and, since $H_{*}(X)\simeq H^{n-*}(\RP^{n}, \RP^{n}\backslash X),$ to compute the Betti numbers of $X.$\\
We first define the table $E=(e^{i,j})_{i,j\in \Z}$ with $e^{i,j}\in \N,$ and whose nonzero part $E'=\{e^{i,j}\, |\,0\leq i\leq 2,\, 0\leq j\leq n\}$ is the following table:
$$
E'=\begin{array}{|c|c|c|}
\hline
1&0&0\\
\hline
\vdots&\vdots&\vdots\\
\hline
1&0&0\\

\hline
c&0&0\\
\hline
0& b_{0}(\Omega^{\mu})-1&d\\
\hline
\vdots&\vdots&\vdots\\

\hline
0&b_{0}(\Omega^{1})-1&b_{1}(\Omega^{1})\\
\hline
\end{array}
$$
 where  $c=e^{0,\mu}$ and we have $(c,d)=(1,b_{1}(\Omega^{\mu}))$ if $w_{1,\mu}=0$ and $(c,d)=(0,0)$ otherwise.
\begin{teo} \label{projgen}If $\mu=n+1$ then $X$ is empty; in the contrary case for every $k\in \Z$ the following formula holds:
$$b_{k}(X)=e^{0,n-k}+e^{1,n-k-1}+e^{2,n-k-2}.$$
Moreover if $i:X\to \RP^{n}$ is the inclusion map and $i_{*}$ is the map induced on homology, then
$$e^{0,n-k}=\emph{\textrm{rk}}(i_{*})_{k}.$$
\end{teo}
The last statement follows from the formula
$$b_{n-k}(\RP^{n})=\textrm{rk}(c^{*})_{n-k}+\textrm{rk}(j_{*})_{k}.$$
\begin{example}[The complex squaring]
 Consider the quadratic forms
       $$q_{0}(x)=x_{0}^{2}-x_{1}^{2},\quad q_{1}(x)=2x_{0}x_{1}.$$ 
 Identifying $\R^{2}$ with $\C$ via $(x_{0},x_{1})\mapsto x_{0}+ix_{1},$ the map $q=(q_{0},q_{1})$ is the complex squaring $z\mapsto z^{2}.$ We
easily see that the common zero locus set of $q_{0}$ and $q_{1}$ in
$\RP^{1}$ is empty. The image of the linear map $\overline q:\R^{2}\to \Q(2)$ defined by $\eta\mapsto \eta q$ consists of
a plane intersecting the set of degenerate forms $Z$ only at the
origin; we identify $\Q(2)$ with the space of $2\times2$ real symmetric
matrices. Thus $\overline q(S^{1})$ is a circle looping around $Z=\{\det =0\}$ and
the index function is constant:
       $$\ii^{+}(\omega q)=1,\quad \omega \in S^{1}.$$
Thus $\Omega^{1}=S^{1};$  and the table $E$ in this case has the
following picture:
$$E=
\begin{array}{|c|c|c}

c&0&0\\
0&0&d\\
\hline
\end{array}
$$
On the other hand the first Stiefel-Whitney class of the bundle $L_{1}\to \Omega^{1}$ in this case is nonzero; hence in this case $(c,d)=(0,0)$ and we have that $b_{k}(X)=0$ for every $k$, as confirmed from the fact that $X=\emptyset.$
\end{example}

Alternatively we could give a direct proof of the previous theorems using Theorems A, B and C of \cite{AgLe}: the reader should recognize in the previous table the structure of some spectral sequence.\\
The previous theorem raises the question: when do we have $w_{1,\mu}\neq0$? Since $\mu=\max \ii^{+},$ then clearly $\Omega=S^{1}$ and $\ii^{+}\equiv \mu.$ Moreover since $\mu=\ii^{+}(\eta)=n+1-\textrm{ker}(\eta Q)-\ii^{+}(-\eta)=n+1-\textrm{ker}(\eta Q)-\mu$ it follows $\mu\leq [\frac{n+1}{2}].$\\
It is interesting to classify pairs of quadratic forms $(q_{0},q_{1})$ such that $\ii^{+}$ is constant; this classification follows from a general theorem on the classification up to congruence of pairs of real symmetric matrices (see \cite{Th}).\\
The formula for the Betti numbers of the spherical part $Y$ is provided in the Appendix 1.

\section{Classical applications}

We discuss here some applications of the previous results; the reader is referred to \cite{Barvinok} for a detailed treatment using different techniques. We start with the following theorem, proved by Calabi in \cite{Calabi}.

\begin{teo}[Calabi] Let $q_{0},q_{1}$ be real quadratic forms over $\R^{n+1}$ with $n+1\geq 3.$ If the only $x\in \R^{n+1}$ satisfying $q_{0}(x)=q_{1}(x)=0$ is $x=0,$ then there exists a real linear combination $\omega q_{0}+\omega_{1}q_{1}$ which is positive definite.
\end{teo}

\begin{proof}The hypothesis is equivalent to $n+1\geq 3$ and $X=\{[x] \in \RP^{n}\, |\, q_{0}(x)=0=q_{1}(x)\}=\emptyset$ and the thesis to $\Omega^{n+1}\neq \emptyset.$\\
First notice that for every $k\geq 2$ we have $b_{1}(\Omega^{k})=0:$ if it was the contrary, then $b_{0}(\Omega^{k})=1=b_{1}(\Omega^{k-1})$ and Theorem \ref{projective} would give $b_{k-1}(\RP^{n}\backslash X)=b_{k-1}(\RP^{n})=b_{0}(\Omega^{k})+b_{1}(\Omega^{k-1})=2,$ which is absurd. Thus if $n+1>2$ we have
$$1=b_{n}(\RP^{n})=b_{n}(\RP^{n}\backslash X)=b_{0}(\Omega^{n+1})+b_{1}(\Omega^{n})=b_{0}(\Omega^{n+1})$$
which implies $\Omega^{n+1}\neq \emptyset.$
\end{proof}

Thus the previous theorem states that for $n+1\geq 3$
$$X=\emptyset \Rightarrow \Omega^{n+1}\neq \emptyset.$$
Also the contrary is true, with no restriction on $n:$ if $X\neq \emptyset$ then $0=b_{n}(\RP^{n}\backslash X)=b_{0}(\Omega^{n+1})+b_{1}(\Omega^{n})$ which implies $\Omega^{n}\neq S^{1}$ and $\Omega^{n+1}=\emptyset.$
Thus we have the following corollary.
\begin{coro}\label{emptyness} If $n+1\geq 3,$ then $X=\emptyset \iff \Omega^{n+1}\neq \emptyset.$\end{coro}
Using the previous we can prove the well known quadratic convexity theorem.
\begin{teo} If $n+1\geq 3$ and $q:\R^{n+1}\to \R^{2}$ is defined by $x\mapsto (q_{0}(x),q_{1}(x)),$ where $q_{0},q_{1}$ are real quadratic forms, then 
$$q(S^{n})\subset \R^{2}\quad \textrm{is a convex set}.$$
\begin{proof}
First observe that if $S^{n}=\{g(x)=1\}$ with $g$ quadratic form, then for a given $c=(c_{0},c_{1})$ we have $S^{n}\cap q^{-1}(c)\neq \emptyset$ iff $S^{n}\cap q_{c}^{-1}(0)\neq \emptyset$ iff $X(q_{c})=\emptyset,$ where $q_{c}$ is the quadratic map whose components are $(q_{0}-c_{0}g, q_{1}-c_{1}g)$ and $X(q_{c})=\{[x]\in\RP^{n}\, |\, q_{c}(x)=0\}.$ Thus by Corollary \ref{emptyness} we have $X(q_{c})\neq \emptyset$ iff $\Omega^{n+1}(q_{c})=\emptyset$ (here $n+1\geq3$).\\
Let now $a=(a_{0},a_{1})$ and $b=(b_{0},b_{1})$ be such that $X(q_{a})\neq\emptyset \neq X(q_{b})$ and suppose there exists $T\in [0,1]$ such that $aT+(1-T)b\notin q(S^{n}).$ Then by Corollary \ref{emptyness} there exists $\eta \in \R^{2}$ such that
$$\eta Q-\langle \eta, aT+(1-T)b\rangle I>0.$$
Assume $\langle \eta, a-b\rangle \geq 0,$ otherwise switch the role of $a$ and $b.$ We have $0<\eta Q -\langle \eta, aT+(1-T)b\rangle I=\eta Q+\langle \eta, T(b-a)\rangle I-\langle \eta, b\rangle I\leq \eta Q-\langle \eta, b\rangle I.$ Thus we got $\eta Q-\langle \eta, b\rangle I>0,$ which implies $\Omega^{n+1}(q_{b})\neq \emptyset,$ but this is impossible by corollary \ref{emptyness} since $X(q_{b})\neq \emptyset.$ Hence for every $t\in[0,1]$ we have $at+(1-t)b\in q(S^{n}).$
\end{proof}
\end{teo}
The conclusions of the previous theorems are false if $n+1=2:$ pick $q_{0}(x,y)=x^{2}-y^{2}$ and $q_{1}(x,y)=2xy,$ then $q_{0}(x)=q_{1}(x)=0$ implies $x=0$ but any real linear combination of $q_{0}$ and $q_{1}$ is sign indefinite. Moreover $q(S^{1})=S^{1}$ which of course is not a convex subset of $\R^{2}.$

\begin{coro}\label{convex}If $q:\R^{n+1}\to \R^{2}$ has homogeneous quadratic components, then $q(\R^{n+1})$ is closed and convex.
\end{coro}
\begin{proof} Since $q(\R^{n+1})$ is the positive cone over $q(S^{n}),$ then it is closed and convex. 
\end{proof}

The previous proof works only for $n+1\geq 3,$ but the theorem is actually true with no restriction on $n.$ The number of quadratic forms is indeed important, as the following example shows: let $q:\R^{3}\to \R^{3}$ be defined by $(x_{0},x_{1},x_{2})\mapsto(x_{0}x_{1},x_{0}x_{2},x_{1}x_{2});$ then the image of $\R^{3}$ under $q$ consists of the four hortants $\{x_{0}\geq 0, x_{1}\geq 0, x_{2}\geq 0\},\,\{x_{0}\leq 0, x_{1}\leq 0, x_{2}\geq 0\},\,\{x_{0}\leq 0, x_{1}\geq 0, x_{2}\leq 0\},\,\{x_{0}\geq 0, x_{1}\leq 0, x_{2}\leq 0\}.$

\section{Homological complexity}
In this section we derive a bound for the homological complexity of $X$ and, in the case $X$ is a regular intersection of quadrics, also a bound for each specific Betti number. 
\begin{teo}\label{bound}Let $X$ be the intersection of two real quadrics in $\RP^n.$ Then
$$b(X)\leq 2n$$
Moreover this bound is sharp.
\end{teo}

We notice that the Universal Coefficients Theorem implies $b(X; \Z)\leq b(X)$ and the previous bound holds also for integer Betti numbers.\\
As a corollary, using the transfer exact sequence with $\Z_{2}$ coefficients for the double covering $Y\to X$  (see \cite{Hatcher}) we have the following\footnote{In fact this sequence yields for every $k\geq0$ the inequality $b_{k}(Y)\leq 2b_{k}(X)$; notice that this inequality is not sharp, as the case $S^{n}\to \RP^{n}$ shows.}. 

\begin{coro}Let $Y$ be the intersection of two real quadrics on $S^n$. Then
$$b(Y)\leq 4n$$
\end{coro}

As suggested by the referee, it might be that the maximal homological complexity is attained by a smooth $X.$ We notice that this property is not true even for the case of one single quadric: the smooth nonempty quadric in $\RP^2$ has total Betti number $2$ (it is homeomorphic to a circle), while the singular quadric given by $x_0x_1=0$ is homeomorphic to the wedge of two circles and its total Betti number is $3.$\\Despite this, the property turns out to be true for the intersection of two quadrics in an \emph{even} dimensional projective space, while it is still false in the \emph{odd} dimensional case.  As we will see in the following example the total Betti number of the intersection of two quadrics in an \emph{odd} dimensional projective space can exceed the homological complexity of the complete intersection.
\begin{example}Consider the quadratic map $q:\R^{n+1}\to \R^2\simeq \C$ given by:
$$q(x_0,\ldots, x_n)=\sum_{k=0}^{n}x_k^2 e^{\frac{2k\pi i}{n+1}}$$
The set $X$ in this case is given by $q_0=q_1=0$, where $q_0,q_1$ are the components of $q$. Notice that $X$ is smooth for \emph{even} $n$ and singular otherwise. An easy computation shows that the index function alternates its values between $\big[\frac{n+2}{2}\big]$ and $\big[\frac{n+2}{2}\big]-1$, each one being assumed $n+1$ times. Thus the table $E'$ of theorem \ref{projgen} in this case is the following:
$$
E'=\begin{array}{|c|c|c|}
\hline
1&0&0\\
\hline
\vdots&\vdots&\vdots\\
\hline
1&0&0\\

\hline
c&0&0\\
\hline
0& n&d\\
\hline
0& 0&1\\
\hline
\vdots&\vdots&\vdots\\

\hline
0&0&1\\
\hline
\end{array}
$$
Since the maximum of the index function is $\mu=\big[\frac{n+2}{2}\big]$, then $b_1(\Omega^\mu)=0$ and consequently $w_{1,\mu}=0$. Thus the pair $(c,d)$ equals $(1,0)$ and
$$b(X)=2n.$$
This example shows that the maximal bound of theorem \ref{bound} is indeed sharp.\\
It is interesting to compare this bound with the so called Smith-Thom inequality (see the Appendix of \cite{Wi}). This inequality states that for a real algebraic variety $Y\subset \RP^n$ with complex part $Y_\C$ the following inequality holds:
\begin{equation}\label{smith}b(Y;\Z_2)\leq b(Y_\C;\Z_2).\end{equation}
In the previous example $X$ was smooth for \emph{even} $n$ and singular otherwise. Since the total Betti number of the complete intersection of two quadrics in $\CP^n$ is known to be $2n-(1+(-1)^{n+1})$, then for every smooth intersection $X'$ of two real quadrics in $\RP^n$ inequality (\ref{smith}) implies:
$$b(X')\leq 2n-2,\quad \textrm{$n$ odd}$$
Thus the previous example shows that the maximal complexity of the intersection of two real quadrics in an odd dimensional projective space exceeds the smooth case, and in particular the same holds for the intersection of two complex quadrics.

\end{example}
Before proving theorem \ref{bound} we discuss the following lemma, that will allow to make a kind of general position argument. 
\begin{lemma}\label{perturb}There exists a positive definite form  $p\in \Q(n+1)$ such that for every $\eps>0$ small enough the set 
$$Z(\eps)=\{\omega \in S^1\, |\, \textrm{ker}(\omega q-\eps p)\neq 0\}$$
consists of a finite number of points and the difference of the index function $\omega \mapsto \ii^{-}(\omega q-\eps p)$ on adjacent components of $S^{1}\backslash Z(\eps)$ is $\pm 1.$
\end{lemma}
\begin{proof}Let $\Q^{+}$ be set of positive definite quadratic forms in $\Q(n+1)$ and consider the map $F:S^{1}\times \Q^{+}\to \Q(n+1)$ defined by 
$$(\omega, p)\mapsto \omega q-p.$$
Let $Z$ be the set of singular quadratic forms in $\Q(n+1),$ semialgebraically and Nash stratified by $Z=\coprod Z_i$.
Since $\Q^{+}$ is open in $\Q,$ then $F$ is a submersion and $F^{-1}(Z)$ is Nash-stratified by $\coprod F^{-1}(Z_{i}).$ Then for $p\in \Q^{+}$ the evaluation map $\omega \mapsto f(\omega)-p$ is transversal to all strata of $Z$ if and only if $p$ is a regular value for the restriction of the second factor projection $\pi:S^{1}\times \Q^{+}\to \Q^{+}$ to each stratum of $F^{-1}(Z)=\coprod F^{-1}(Z_{i}).$
Thus let $\pi_{i}=\pi|_{F^{-1}(Z_{i})}:F^{-1}(Z_{i})\to \Q^{+};$ since all data are smooth semialgebraic, then by semialgebraic Sard's Lemma (see \cite{BCR}), the set $\Sigma_{i}=\{\hat{q}\in \Q^{+}\, | \, \hat{q}\textrm{ is a critical value of $\pi_{i}$}\}$ is a semialgebraic subset of $\Q^{+}$ of dimension strictly less than $\dim (\Q^{+}).$ Hence $\Sigma=\cup_{i}\Sigma_{i}$ also is a semialgebraic subset of $\Q^{+}$ of dimension $\dim (\Sigma)<\dim (\Q^{+})$ and for every $p\in \Q^{+}\backslash \Sigma$ the map $\omega\mapsto f(\omega)-p$ is transversal to each $Z_{i}.$ Since $\Sigma$ is semialgebraic of codimension at least one, then there exists $p\in \Q^{+}\backslash \Sigma$ such that $\{t p\}_{t>0}$ intersects $\Sigma$ in a finite number of points, i.e. for every $\eps>0$ sufficiently small $\eps p\in \Q^{+}\backslash \Sigma.$ Since the codimension of the singularities of $Z$ are at least three, then for $p\in \Q^{+}\backslash \Sigma$ and $\eps>0$ small enough the set $\{\omega \in S^{1}\, |\,\ker(\omega q-\eps p)\neq0\}$ consists of a finite number of points. Moreover if $z$ is a smooth point of $Z$, then its normal bundle at $z$ coincides with the sets of quadratic forms $\{\lambda (x\otimes x)\,|\, x\in \ker(z)\}_{\lambda \in \R}$ then also the second part of the statement follows.\end{proof}

Essentially lemma \ref{perturb} tells that we can perturb the map $\omega\mapsto \omega q$ in such a way that crossing each point where the determinant vanishes the index function changes exactly by $\pm1;$ proposition \ref{union} tells us how to control the topology of the sets $\Omega^{j+1}$ after this perturbation. 
\begin{proof}(Theorem \ref{bound}) Because of theorem \ref{projgen} we have:
$$b(X)\leq n+1-2(\mu-\nu)+\sum_{k=\nu}^{\mu-1}b_0(\Omega^{k+1})$$
where $\nu=\min \ii^+|_\Omega$ (essentially we have made the sum of the elements of the table $E'$ before theorem \ref{projgen} taking into account that for $k<\nu$ the set $\Omega^{k}$ consists of the whole $S^1$). By lemma \ref{perturb} there exists a positive definite form $p$ such that for every $\eps>0$ sufficiently small the set $Z(\eps)=\{\omega \in S^{1}\, |\, \textrm{ker}(\omega q-\eps p)\neq 0\}$ consists of a finite number of points; moreover by lemma \ref{union} for such a $p$ and for $\eps>0$ small enough we also have the equality $b_0(\Omega^{j+1})=b_0(\Omega_{n-j}(\eps))$. This in particular gives
$$b(X)\leq n+1-2(\mu-\nu)+\sum_{k=\nu}^{\mu-1}b_0(\Omega_{n-k}(\eps)).$$
Since for each $\nu\leq k\leq \mu-1$ the set $\Omega_{n-j}(\eps)$ is a disjoint union of closed intervals of $S^{1}$, then:
$$b_0(\Omega_{n-k}(\eps))=\frac{1}{2}b_{0}(\partial \Omega_{n-j}(\eps)).$$
In particular $\sum b(\Omega_{n-k}(\eps))$ equals $\frac{1}{2}\sum b_{0}(\partial \Omega_{n-k}(\eps)),$ where in both cases the sum is made over the indices $\nu\leq k\leq \mu-1$. The second part of lemma \ref{perturb} implies now that each one of the points in $Z(\eps)$ belongs to the boundary of exactly one of the $\Omega_{n-k}(\eps), \nu\leq k\leq \mu-1.$ This implies that the previous sum $\frac{1}{2}\sum b(\partial \Omega_{n-j}(\eps))$ equals exactly half the number of points of $Z(\eps).$ On the other hand $Z(\eps)$ is defined by the intersection in $\R^2$ of the unit circle with a curve of degree $n+1$, namely $\det(\omega q-\eps p)=0.$ Thus $Z(\eps)$ has at most $2(n+1)$ points and 
$$b(X)\leq n+1-2(\mu-\nu)+n+1.$$ 
Since for $\mu=\nu$ we have $b(X)\leq n+1$, then we may assume $\mu-\nu\geq 1$ which finally gives the desired inequality.
\end{proof}
\begin{remark} 
Using the same technique, it is possible to prove that if $X$ is the intersection of three real quadrics in $\RP^n$, then $b(X)\leq n(n+1).$ The interested reader is referred to \cite{Le3}.
\end{remark}
We move now back to the smooth case and we prove the bound for each specific Betti number. The interesting part is that our bound $b_k(X)\leq2(k+2)$ does not depend on $n$.
We start with the following lemma.
\begin{lemma}\label{bk}Let $X$ be a nonsingular intersection of two quadrics in $\RP^n$ and $S^1$ be the union of two half circles $C_1$ and $C_2=-C_1$ such that $C_1\cap C_2=\{\eta, -\eta\}$ and $\det(\eta q)\neq 0$.Then for every $k\in \N$ we have:
$$b_k(X)\leq b_0(\Omega^{n-k}\cap C_1)+b_0(\Omega^{n-k}\cap C_2)$$
\begin{proof}
We first notice that by the semialgebraic Mayer-Vietoris sequence (see \cite{BCR})  for the pair $(C_1,C_2)$ we have:
$$b_0(\Omega^{n-k})\leq b_0(\Omega^{n-k}\cap C_1)+b_0(\Omega^{n-k}\cap C_2)-b_0(\Omega^{n-k}\cap \{\eta\})-b_0(\Omega^{n-k}\cap \{-\eta\})$$
If $k\neq n- \mu$ then theorem \ref{projgen} gives the desired bound, since $b_k(X)\leq b_0(\Omega^{n-k}).$ In the case $k=n-\mu$ theorem \ref{projgen} gives $b_k(X)\leq b_0(\Omega^{\mu})+b_1(\Omega^{\mu-1}).$ If $b_1(\Omega^{\mu-1})=0$ then again the result follows from the previous inequality. In the remaining case the index function has only two values $\mu$ and $\mu-1$ and either $\eta$ or $-\eta$ has index $\mu$: if it was not the case then they both would have index $\mu-1$ which is impossible because $\mu=\max \ii^+$ and $\ii^+(\eta)=n+1-\ii^+(-\eta).$ Thus also in this case the conclusion follows from the previous inequality.\end{proof}
\end{lemma}

\begin{teo}\label{bkbound}Let $X$ be a \emph{smooth} intersection of two quadrics in $\RP^n$. Then for every $k\in \N$:
$$b_k(X)\leq 2(k+2)$$
\begin{proof}
We start by proving a formula for the index function of a smooth intersection of quadrics. Recall that this amounts to the fact the set of points $Z$ on $S^1$ where the determinant of $\omega q$ vanishes is finite and at each one of these points the index function jumps exactly by $\pm1.$ Let us fix an orientation of the circle and divide the set $Z$ into the disjoint subset $Z^+$ and $Z^-$ consisting respectively of the points where the index function jumps by $+1$ and $-1$ when crossing them counterclockwise. Let us fix also two points $\eta, -\eta$ on $S^1$ where the determinant does not vanish. Consider the counterclockwise arc $I=[-\eta, \eta]$ and a point $\omega\in \textrm{Int}(I)$ such that $\det(\omega q)\neq 0.$ Le $\rho^{\pm}(\omega)$ and $\lambda^{\pm}(\omega)$  be respectively the cardinality of $[-\eta, \omega]\cap Z^\pm$ and $[\omega, \eta]\cap Z^\pm.$ Let also $2\theta$ be the number of non real projective solutions to $\det(\omega q)=0.$ \\ We claim that using this notations the index function verifies:
\begin{equation}\label{index}\ii^+({\omega})=\rho^+(\omega)+\lambda(\omega)+\theta.\end{equation}
Let us prove this formula. We have $\ii^+(-\omega)=n+1-\ii^+(\omega)$, because $\det(\omega q)\neq 0$, and $\ii^+(-\omega)=\ii^+(\omega)+\lambda^+(\omega)-\lambda(\omega)+\rho(\omega)-\rho^+(\omega)$, because of the definition of the functions $\lambda^\pm$ and $\rho^\pm$ and the fact that $Z^-=-Z^+$. Thus:
$$2\ii^+(\omega)=n+1-\lambda^+(\omega)+\lambda(\omega)-\rho(\omega)+\rho^+(\omega)$$
On the other hand we also have:
$$\lambda^+(\omega)+\lambda(\omega)+\rho(\omega)+\rho^+(\omega)+2\theta=n+1$$
which combined with the previous equation gives (\ref{index}).\\
Because of lemma \ref{bk}, to prove the theorem it is sufficient to show that 
$$b_0(\Omega^{n-k}\cap I)\leq k+2$$
If $\Omega^{n-k}$ is not empty let $\eta_1, \eta_2$ be two points satisfying: (i) $I=[-\eta, \eta_1]\cup[\eta_1, \eta_2]\cup[\eta_2,\eta]$; (ii) the determinant does not vanish at $\eta_1, \eta_2$; (iii) $\ii^+(\eta_1)=\ii^+(\eta_2)=n-k$; (iv) either $\ii^+|_{[-\eta, \eta_1]}\leq n-k$ or $\ii^+|_{[-\eta, \eta_1]}\geq n-k$ and the same holds for $\ii^+|_{[\eta_2, \eta]}$. To get these points it is sufficient to let $\eta_1$ belong to the closest interval to $-\eta$ where the index function is $n-k$ and $\eta_2$ to the closest interval to $\eta$ where the index function is $n-k$. To simplify notations let:
$$a^{\pm}=\textrm{Card}( [-\eta, \eta_1]\cap Z^\pm),\quad b^{\pm}=\textrm{Card}([\eta_1, \eta_2]\cap Z^\pm),\quad c^{\pm}=\textrm{Card}([\eta_2,\eta]\cap Z^\pm)$$
Then by the index formula (\ref{index}) we have $n-k=\ii^+(\eta_1)=c^-+b^-+a^++\theta$ and  $n-k=\ii^+(\eta_2)=c^-+b^++a^++\theta$, which combined give $b^-=b^+$. On the other hand we also have $a^-+a^++2b+c^-+c^+=n+1-2\theta$, which combined with $n-k=a^++b^-+c^-+\theta$ gives:
$$b\leq k+1-\theta\leq k+1$$
Thus the number of points in $[\eta_1, \eta_2]$ where the index function can change sign are at most $k+1$ and in this case:
$$b_0(\Omega^{n-k}\cap I)= b_0(\Omega^{n-k}\cap [\eta_1, \eta_2])\leq k+2$$
\end{proof}
\end{teo}
\begin{remark}Notice that because of the Universal Coefficients Theorem, the result holds also for Betti numbers with $\Z$ coefficients.
\end{remark}
\begin{remark}Using the transfer exact sequence for the covering $Y\to X$ as above, we also get $$b_k(Y;\Z)\leq b_k(Y)\leq 4(k+2)$$ for $Y$ a nonsingular intersection of two quadrics on $S^n.$

\end{remark}
\begin{remark}
Estimates on the Betti Numbers of system of quadratic inequalities are given in the general case in \cite{BaKe} and \cite{BaPaRo}; in the case of intersection of quadrics in projective space estimates on the number of connected components are given in \cite{DeItKh}. In particular, following the notations of \cite{DeItKh}, we can denote by $B_{r}^{k}(n)$ the maximum value that the $k$-th Betti number of a complete intersection of $r+1$ quadrics in $\RP^{n}$ can have. There it is proved that
$$B_{2}^{0}(n)\leq \frac{3}{2}l(l-1)+2,\quad l=[n/2]+1.$$
With this notation our previous result reads:
$$B_1^k(n)\leq 2(k+2)$$
In particular the maximal number of connected component is $4$; this number is sharp in the family of nonsingular intersection, as the case of two appropriate quadrics in $\RP^2$ shows. More generally, for \emph{even} $n$ the pair of quadrics of the example at the beginning of this section attains the maximum for $k=\frac{n}{2}-1$
\end{remark}
\section*{Appendix: Level sets of quadratic maps}
In this section we prove formulas similar to that in theorem \ref{projgen} for the intersection of two quadrics on the sphere. More generally we give a formula for the level set in $\R^{n+1}$ of a homogeneous quadratic map. The reason for putting this theorems in the Appendix is that we use the machinery of spectral sequences.\\
We start with the formula for the Betti numbers of the set defined by
$$Y=\{x\in S^{n}\,|\, q(x)\in K\}.$$
\begin{teo}\label{spherical}The following formula holds for $k< n-2$: $$\tilde{b}_{k}(Y)=\tilde{b}_{n-k-1}(S^{n}\backslash Y)=b_{0}(\Omega^{n-k},\Omega^{n-k+1})+b_{1}(\Omega^{n-k-1},\Omega^{n-k}).$$
\end{teo}

\begin{proof}
The first equality follows from Alexander duality. For the second consider the set
$$B=\{(\omega, x)\in\Omega\times S^{n}\,|\, (\omega q)(x)>0\}.$$
The projection $p_{2}:B\to S^{n}$ gives a homotopy equivalence $B\sim p_{2}(B)=S^{n}\backslash Y$ (the fibers are contractible). On the other side for $\epsilon>0$  sufficiently small the inclusion
$$B(\eps)=\{(\omega,x)\in \Omega\times S^{n}\,|\, (\omega q)(x)\geq \eps\}\hookrightarrow B$$ 
is a homotopy equivalence. Consider $\pi=p_{1}|_{B(\eps)}:B(\eps)\to \Omega$ and the Leray spectral sequence associated to it: 
$$(E_{r}(\eps),d_{r})\Rightarrow H^{*}(B(\eps);\Z_{2}),\,\,E_{2}(\eps)^{i,j}=\check{H}^{i}(\Omega, \mathcal{F}^{j}(\eps)),$$ where $\mathcal{F}^{j}(\eps)$ is the sheaf associated to the presheaf $V\mapsto H^{j}(\pi^{-1}(V)).$ Since $B(\eps)$ and $\Omega$ are locally compact and $\pi$ is proper ($B(\eps)$ is compact) then the following isomorphism holds for the stalk of $\mathcal{F}^{j}(\eps)$ at each point $\omega\in \Omega:$
$$\mathcal{F}^{j}(\eps)_{\omega}\simeq H^{j}(\pi^{-1}(\omega)).$$
Let $g\in \R[x_{0},\ldots, x_{n}]_{(2)}$ such that $S^{n}=\{g(x)=1\}$, then $\pi^{-1}(\omega)\simeq \{x\in S^{n}\,|\, (\omega q-\eps g)(x)\geq 0\}$ has the homotopy type of a sphere of dimension $n-\textrm{ind}^{-}(\omega q-\eps g);$ thus if we set $\ii^{-}(\eps)$ for the function $\omega \mapsto \textrm{ind}^{-}(\omega q-\eps g),$ we have that for $j>0$ the sheaf $\mathcal{F}^{j}(\eps)$ is locally constant with stalk $\Z_{2}$ on $\Omega_{n-j}(\eps)\backslash \Omega_{n-j-1}(\eps),$ where $\Omega_{n-j}(\eps)=\{\ii^{-}(\eps)\leq n-j\},$ and zero on its complement. Since $\Omega_{n-j-1}(\eps)$ is closed in $\Omega_{n-j}(\eps),$ we have for $j>0:$
$$\check{H}^{i}(\Omega, \mathcal{F}^{j}(\eps))=\check{H}^{i}(\Omega_{n-j}(\eps), \Omega_{n-j-1}(\eps)).$$
Since the sets $\{\Omega_{n-j}(\eps)\}_{j\in \N}$ are CW-subcomplex of the one-dimensional complex $S^{1}$ (covers such that triple intersections of their open sets are empty are cofinal), then $E_{2}^{i,j}(\eps)=0$ for $i\geq2$ and the Leray spectral sequence of $\pi$ degenerates at $E_{2}(\eps).$ By semialgebraicity the topology of $\Omega_{n-j}(\eps)$ is definitely constant in $\eps$ and form small $\eps$ we have
$$E_{2}^{i,j}(\eps)\simeq \varprojlim\{\check{H}^{i}(\Omega_{n-j}(\eps), \Omega_{n-j-1}(\eps))\},\quad j>0.$$
Lemma \ref{union} implies for $j>0$ the isomorphism $E_{2}^{i,j}(\eps)\simeq \check{H}^{i}(\Omega^{j+1},\Omega^{j+2})$ and the conclusion follows.
\end{proof}

\begin{remark}The anomalous behaviour for $j=0$ is due to the fact that there is no canonical choice for the generator of $H^0(S^0)=\Z_2\oplus\Z_2.$
\end{remark}

We discuss now the topology of the level sets of a homogeneous quadratic map. We start with the following observation: in the case we are given a semialgebraic subset $A$ in $\R^{n}$ defined by inequalities involving polynomials of degree two (the presence of degree one polynomials reduce to this case by restricting to affine subspaces), then  $A$ is homotopy equivalent to the set $A(\eps)$ defined in the projective space by homogenization of the inequalities defining $A$ and by adding the inequality $l_{\eps}\leq 0$ for $\eps>0$ small enough (see \cite{Le}).\\
Thus we reduce the problem of studying the topology of $A$ to that of studying the set of projective solutions of a system of \emph{three} homogeneous quadratic inequalities, one of which is a very particular fixed one. A technique to deal with such a problem, which is a generalization of the one discussed in this paper, is introduced in \cite{AgLe}.\\
A case of particular interest is when $A$ is the level set of a homogeneous quadratic map:
$$A=q^{-1}(c),\quad c\in \R^2$$
It is convenient in this case to use the geometry of the \emph{negative} inertia index and define the sets:
$$C=\{\omega\in S^1\,|\, \langle\omega, c\rangle<0\}\quad \textrm{and}\quad C_k=\{\omega \in C\,|\, \ii^-(\omega q)\leq k\}\quad k\in \N$$
With this notation we have the following theorem; for a proof the interested reader is referred to \cite{Le}.
\begin{teo}Let $q:\R^n\to \R^2$ be a homogeneous quadratic map and $c\in \R^2.$ Then $A=q^{-1}(c)$ is nonempty if and only if $\min \ii^-|_C\neq 0$, in which case for $0\leq k\leq n$ we have the following formula:
$$\tilde b_{k}(A)=b_{0}(C_{k+1},C_{k})+b_{1}(C_{k+2},C_{k+1})$$
\end{teo}
Notice in particular that if $c=0$, then $A$ is contractible: it is the cone over $A\cap S^{n-1}$ and for the Betti numbers of this set we can use theorem \ref{spherical}.
\begin{remark}
The statement of the previous theorem still holds for systems of \emph{inequalities}:
if $A=\{q_{0}\leq c_{0},q_{1}\leq c_{1}\}$ then the result is the same by setting $K=\{y_0\leq 0, y_1\leq 0\}$ and $C_{k}=\{\omega \in K^{\circ}\cap S^{1}\, |\, \langle\omega, c\rangle <0, \, \ii^{-}(\omega q)\leq k\}.$
\end{remark}

\end{document}